\documentclass[11pt, reqno]{amsart}
\usepackage{amsmath,amssymb,amsfonts,amscd,hyperref,color}
\usepackage[latin1]{inputenc}

\usepackage[abs]{overpic}
\usepackage{verbatim}

\newcommand{\Hmm}[1]{\leavevmode{\marginpar{\tiny%
$\hbox to 0mm{\hspace*{-0.5mm}$\leftarrow$\hss}%
\vcenter{\vrule depth 0.1mm height 0.1mm width \the\marginparwidth}%
\hbox to
0mm{\hss$\rightarrow$\hspace*{-0.5mm}}$\\\relax\raggedright #1}}}

\newtheorem{thm}{Theorem}[section]
\newtheorem{cor}[thm]{Corollary}

\newtheorem{lemma}[thm]{Lemma}
\newtheorem{pro}[thm]{Proposition}

\theoremstyle{definition}
\newtheorem*{defi}{Definition}

\newtheorem{rem}[thm]{Remark}

\numberwithin{equation}{section}
\newcommand{\Z}{{\mathbb Z}}
\newcommand{\R}{{\mathbb R}}

\newcommand{\N}{{\mathbb N}}

\let\H\undefined
\let\L\undefined
\newcommand{\H}{H}
\newcommand{\L}{L}
\newcommand{\F}{F}

\newcommand{\ph}{{\varphi}}

\newcommand{\Om}{{\Omega}}

\newcommand{\lm}{{\lambda}}

\newcommand{\ow}[1]{\widetilde{ #1}}

     \def\Gd{\Delta}

\begin{document}
\title[Optimal Hardy inequalities on graphs]
{Optimal Hardy inequalities for Schr\"odinger operators on graphs}
\author[M.~Keller]{Matthias Keller}
\address{M.~Keller,  Institut f\"ur Mathematik, Universit\"at Potsdam
\\14476  Potsdam, Germany}
\email{mkeller@math.uni-potsdam.de}
\author[Y.~Pinchover]{Yehuda Pinchover}
\author[F.~Pogorzelski]{Felix Pogorzelski}
\address{Y.~Pinchover and F.~Pogorzelski, Department of Mathematics, Technion-Israel Institute of Technology, 32000 Haifa, Israel}
\email{pincho@technion.ac.il}
 \email{felixp@technion.ac.il}
\date{\today}
\begin{abstract}
For a given subcritical discrete Schr\"odinger operator $H$ on a  weighted infinite graph $X$, we construct a Hardy-weight $w$ which is {\em optimal} in the following sense. The operator
$H - \lambda w$ is subcritical in $X$ for all $\lambda < 1$, null-critical in $X$ for $\lambda = 1$, and supercritical near any neighborhood of infinity in $X$ for any $\lambda > 1$.
Our results rely on a criticality theory for Schr\"odinger operators on general weighted graphs.
\\[2mm]
\noindent  2000  \! {\em Mathematics  Subject  Classification.}
Primary  \! 39A12; Secondary  31C20, 31C25, 35B09, 35P05, 35R02, 47B39.\\[2mm]
\noindent {\em Keywords.} Green function, Ground state, Hardy inequality, Positive solutions, Discrete Schr\"odinger operators, Weighted graphs.
%%%%%%%%%%%%%%%%%%%
\end{abstract}
%\tableofcontents
%%%%%%%%%%%%%%%%%%%%%%%
\maketitle
% \frenchspacing
%%%%%%%%%%%%%%%%%%%%%%%%%%%%%%%%%%%%%%%%%%%%%%%%%%%%%%%%%%%%
% ABSTRACT
%%%%%%%%%%%%%%%%%%%%%%%%%%%%%%%%%%%%%%%%%%%%%%%%%%%%%%%%%%%%
 %%%%%%%%%%%%% SET UP%%%%%%%%%%%%%%%%%%%%%%%
\setcounter{section}{-1}
\section{Introduction}
In 1921 Landau wrote a letter to Hardy which includes a proof of the following
inequality
\begin{align}\label{CHI}
 \sum_{n=0}^{\infty}|\ph(n)-\ph(n+1)|^{p}\ge
C_p\sum_{n={1}}^{\infty}w(n)|\ph(n)|^{p}
\end{align}
for all finitely supported $\ph:\N_0\to\R$ such that
$\ph(0)=0$ ,where
$$w(n):=\frac{1}{n^{p}}, \quad\mbox{and }
C_p := \left(\frac{p-1}{p}\right)^{p}, $$
with $ n\ge1 $, $ 1<p<\infty $ and $ \N_0:= \{0,1,2,3,\ldots\}. $
 This inequality
 was stated before by Hardy, and therefore, it is called a \emph{Hardy
inequality} (see \cite{KMP06} for a marvelous
description on the prehistory of the celebrated Hardy inequality). Since then Hardy-type inequalities have received an
enormous amount of attention.

By a Hardy-type inequality for a {\em nonnegative} operator $P$ we roughly mean that the inequality $P\geq Cw$ holds for some ``large'' weight function $w$ and optimal constant $C>0$. One particular focus in the literature lies on finding the sharp
constant $C$ to a prescribed Hardy-weight $w$ which is typically
an inverse square weight. For the
classical literature we refer here to the monographs
\cite{BEL15,OK90} and to references therein. Recent developments
include relationships with other functional inequalities, Hardy-type inequalities related to different boundary conditions
\cite{AY,DEFT,KL}, Hardy-type inequalities for fractional Schr\"odinger operators \cite{BD,FS08,FS10,LS}, and
Hardy inequalities for the Laplacian on metric trees \cite{EFK08,NS}.

\medskip

A conceptually different approach was taken by \cite{DFP} in the
context of general elliptic Schr\"odinger operators $P$ in the case
$p=2$. There, the weight function $w\ge0$ is intrinsically derived
from $P$ in terms of superharmonic functions such as the Green
function. This weight $w$ is shown to be optimal in three  aspects:
\begin{itemize}
  \item  For every $\tilde{w}\gneqq w$, the Hardy inequality fails (``Criticality'').

  \item  The  ground state is not an eigenfunction (``Null-criticality'').
    \item   For any $\lm>0$, the Hardy inequality outside of any compact set
   fails for the weight $(1+\lm)w$ (``Optimality near
   infinity'').
\end{itemize}
Such a weight $w$ is called an \emph{optimal Hardy-weight} (for a
rigorous  definition see Section~\ref{s:setup}). Using a different approach, the main results of \cite{DFP} were
later generalized to the $p$-Laplacian for $1<p<\infty$ \cite{DP}.

In the present paper we follow the approaches in \cite{DFP,DP} in the context of
Schr\"odinger operators on weighted graphs (with $p=2$). Similarly to \cite{DFP,DP}, we prove a Hardy inequality with an optimal Hardy weight under
the assumption of the existence of certain positive (super)harmonic functions. There are
two versions of the result, one for bounded positive (super)harmonic functions
and one  for unbounded positive (super)harmonic functions.

Let us present two special cases of our results in the context of
graphs with standard weights and bounded vertex degree. Let $X$ be
a countably infinite vertex set. We denote $x\sim y$ whenever two vertices $x,y$ are
connected by an edge in which case we call $x$ and $y$ {\em adjacent}. The
 {\em degree} $\deg(x)$ of a vertex $x\in X$ is the number of vertices
adjacent to $x$. A function $u:X\to\R$ is said to be {\em harmonic}  on $ W\subseteq X $ if
\begin{align*}
    \Delta u(x):=\sum_{y\sim x}(u(x)-u(y))=0 \qquad x\in W.
\end{align*}
Correspondingly, a function $ u $ is called \emph{superharmonic} on $ W $ if $ \Delta u(x)\ge 0 $
for all $ x \in W $.

Recall that a function $ u:X\to  \R $ is called \emph{proper} on $ W\subseteq X $
if $ u^{-1}(I) $ is compact, (i.e., finite) for all compact $ I\subseteq u(W):=\{u(x)\mid x\in W\} $.

%%%%%%%%%%%%%%%%%%%%%%%%%%%%%%%%%%
\begin{thm}\label{t:graph1} Let a connected graph $X$ with bounded vertex degree and a finite set $ K\subseteq  X $ be given, and let $u:X\to (0,\infty)$ be an unbounded  positive function which is harmonic and proper
on $ X\setminus K $ and such that $ u=0 $ on $ K $.
Then the following Hardy-type inequality holds true
\begin{align*}
\sum_{x,y\in X,x\sim y}(\ph(x)-\ph(y))^{2}\ge\sum_{x\in
    X\setminus K}w(x)\ph(x)^{2}
\end{align*}
for all finitely supported functions $\ph$ with support in $ X\setminus K $,  where  the weight function $w$ is given by
\begin{align*}
    w(x):=\frac{1}{2 u(x)}\sum_{y\sim x}\left(
{u(x)}^{1/2}-{u(y)}^{1/2}
    \right)^{2} \qquad x\in X\setminus K.
\end{align*}
Moreover, $w$ is an optimal Hardy weight in $X\setminus K$.
\end{thm}
%%%%%%%%%%%%%%%%%%%%%%%%%%%%%%%%
In various cases it is hard to find explicit non-trivial harmonic
functions. However, for transient graphs there are plenty of
positive superharmonic functions in terms of the positive minimal Green function
\begin{align*}
    G(x,y):=\sum_{n=0}^{\infty}p_{n}(x,y)
\end{align*}
where $p_{n}(x,y)$ are the matrix elements of the $n$-th power of
the transition matrix given by the matrix elements
$p_{1}(x,y)=1/\deg(x)$, $x,y\in X$. The case when the sum converges
for all $x,y\in X$ is called transient. In this case, $u=G(o,\cdot)$
is known to be a strictly positive superharmonic and harmonic
outside of $o\in X$ (cf e.g. \cite{Woe-Book}). By the minimality of $G$ it follows that $\inf G(o,\cdot)=0$. We prove the following Hardy inequality.
%%%%%%%%%%%%%%%%%%%%%%%%%%%%%%
\begin{thm}\label{t:graph2} Let a transient connected graph with bounded vertex degree
be given, and let $o\in X$ be fixed.
Let $G:X\to(0,\infty)$ be the positive minimal Green function with a pole at $o$, and assume that $G$ is proper.
Then the following Hardy-type inequality holds true
\begin{align*}
    \sum_{x,y\in X,\,x\sim y}(\ph(x)-\ph(y))^{2}\ge\sum_{x\in
    X}w(x)\ph(x)^{2}
\end{align*}
for all finitely supported functions $\ph$ on $X$,  where
\begin{align*}
    w(x):= \frac{\Gd \Big(G(x)^{{1}/{2}}\big)}{G(x)^{{1}/{2}}}
\end{align*}
 is an optimal Hardy weight in $X$, and for all $  x\neq o $
 \begin{align*}
 w(x) =
 \frac{1}{2G(x)}\sum_{y\sim x}\left(
 {G(x)}^{1/2}- {G(y)}^{1/2} \right)^{2}.
 \end{align*}
\end{thm}
%%%%%%%%%%%%%%%%%%%%%%%%%%%%%%%%%%%%%%%%%%%%%%%
These theorems are special cases of Theorem~\ref{thm_optimal_H} and
Corollary~\ref{thm_optimal_H_bounded} which are the main results of
the paper. In Section~\ref{sec_pf_main} we show how they can be
derived from the main results.

As for the proofs of our theorems, we are faced with the challenge
that the two approaches of \cite{DFP,DP} both rely heavily on the chain rule.
However, it is well known that, in general, a chain rule is  not valid
in non-local settings such as graphs. The remedy is twofold. Firstly
it was observed in \cite{BHLLMY} that an analogue of the chain rule
holds for the square root. This is crucial since the square
root of a positive superharmonic function is superharmonic. Secondly, we
rely on a coarea formula that is inspired by the treatment in
\cite{DP} (see also \cite{DFP}). This is in line with the meta-strategy that local
estimates should be replaced by integrated estimates in the case of
graphs.

The paper is structured as follows. In the following section the basic
setting is introduced and the main results are stated. Then in Section~\ref{s:toolbox}, we provide the major tools for the proof which consists of a discrete chain rule, the ground state transform and a co-area formula which are backbone of the proof of our main result. In Sections~\ref{s:crit},~\ref{s:nullcrit} and~\ref{s:optimal}
the proofs of the criticality, null-criticality and optimality near infinity, of our
Hardy weights are respectively presented for the case of Laplace-type operators. In
Section~\ref{sec_pf_main} we deduce these results for Schr\"odinger
operators from the previously proven results using the ground state
transform. Finally, in Section~\ref{s:examples} we present some
first basic examples, where our results can be applied to concrete
graphs.

\section{Set up and main result}
%%%%%%%%%%%%%%%%
\subsection{Graphs, formal Schr\"odinger operators and forms}\label{Subsection-Formal}

 A graph over an infinite discrete set  $X$ is a symmetric function $b:X\times
X\to[0,\infty)$ with zero diagonal such that
$$\sum_{y\in  X}b(x, y)<\infty \quad \textup{ for all } x\in X.$$
We call the elements of $X$  \emph{vertices}. We say that $x,y\in X$ are
\emph{adjacent} or \emph{neighbors} or \emph{connected by an edge}
if $b(x,y)>0$ in which case we write $x\sim y$.

Throughout the paper we assume that
$X$ is a countably infinite set equipped with the discrete topology and $ b $ is a  graph over $ X $. Furthermore we assume that $b$ is \emph{connected}, that is, for every $x$ and $y$ in $X$
there are $x_{0},\ldots,x_{n}$ in $X$ such that $x_{0}=x$, $x_{n}=y$
and $x_{i}\sim x_{i+1}$ for $i=0,\ldots,n-1$.
%%%%%%%%%%%%%%%%%%%%%%%%

For $W\subseteq X$, we denote by $C(W)$ (resp., $C_c(W)$) the space
of real valued functions on $W$ (resp., with compact support in
$W$). By extending functions by zero on $X\setminus W$ the space
$C(W)$  will be considered as a subspace of  $C(X)$.

We say that $f\in C(W)$ is \emph{positive} in $W$ if  $f\geq 0$ and $f\neq 0$ in
$W$, in this case, we also use the notation $f\gneqq 0$.

% We denote the space of $m$-square summable functions by $\ell^{2}(X,m)$ denote the space of bounded functions by $\ell^{\infty}(X)$.

Given a  graph $b$ over $X$, we introduce the associated \emph{formal
Laplacian} ${\L}=\L_{b}$ acting on the space
\begin{align*}
{\F}(X):=\{f\in C(X) \mid \sum_{y\in X}b(x,y)|f(y)|<\infty\mbox{ for
all } x\in X\},
\end{align*}
by
\begin{align*}
{\L} f(x):=\sum_{y\in X}b(x,y)(f(x)-f(y)).
\end{align*}
By the summability assumption on $b$ we have $\ell^\infty(X) \subset
F(X)$.

For a potential $q:X\to\mathbb{R}$, we define the \emph{formal
Schr\"odinger operator} ${\H}$ on $\F(X)$ by
\begin{align*}
{\H}:={\L}+q.
\end{align*}
The associated \emph{bilinear form} $h$ of $H$ on $C_{c}(X)\times C_{c}(X)$ is given by
\begin{align*}
h(\ph,\psi)\!:=\!\frac{1}{2}\sum_{x,y\in
X}\!\!\!b(x,y)(\ph(x)\!-\!\ph(y))(\psi(x)\!-\!\psi(y)) \!+\!\sum_{x\in
X}\!q(x)\ph(x)\psi(x).
\end{align*}
We denote by $h(\ph):=h(\ph,\ph)$ the induced quadratic form on $C_c(X)$, and write $h\ge0\mbox{ on }C_{c}(X)$ (or in short $h\geq 0$)
if $h(\ph)\ge0$ for all $\ph\in C_{c}(X)$.
\medskip

Any function $ w:X\to \R $ gives rise to a canonical quadratic form on $ C_{c}(X) $ which we denote (with a slight abuse of notation) by $w$. It acts as
\begin{align*}
w(\ph):=\sum_{x\in X}w(x)\ph(x)^{2}.
\end{align*}

We denote by
 $\ell^{2}(X)$ the Hilbert space of square summable functions equipped with the
 scalar product
 \begin{align*}
 \langle{f},g\rangle:=\sum_{X}f g=\sum_{x\in X}f(x)g(x) \qquad f,g \in \ell^{2}(X).
 \end{align*}

Finally, we introduce the notion of (super)harmonic functions on a graph.
\begin{defi}
	We say that a function $u$ is $H$-\emph{(super)harmonic} on
	$W\subseteq X$ if $u\in {\F}(X)$ and ${\H}u=0$, ($Hu\ge0$) on $W$.
	We write
	\begin{align*}
	{\H}\ge0 \quad \mbox{on } W
	\end{align*}
	if there exists a positive $H$-superharmonic function $u$ on $W$.
\end{defi}

%%%%%%%%%%%%%%%%%%%
\subsection{Critical Hardy-weights}\label{s:setup}
%%%%%%%%%%%%%
In this subsection we define the notions of criticality, subcriticality
and null-criticality that are fundamental for the present paper. These notions are discussed in detail in \cite{KePiPo1} to which we also refer for references.

\begin{defi}[Critical/subcritical]Let $h$ be a quadratic form associated with the formal
Schr\"odinger operator ${\H}$, such that $h\ge0$ on $C_{c}(X)$. The
form $h$ is called \emph{subcritical} in $X$ if there is a positive
$w\in C(X)$ such that $h-w\ge0$ on $C_{c}(X)$. Otherwise, the form
$h$ is called \emph{critical}  in $X$.
\end{defi}
%%%%%%%%%%%%%%%%%%%%%%%%%%%%%%
In \cite[Theorem~5.3]{KePiPo1} a characterization of criticality is presented. There it is shown, that for a form $ h $ being critical is equivalent to the existence of a
unique positive $H$-superharmonic function $ v $ (which is in fact $H$-harmonic), and which is called the \emph{(Agmon) ground state} of $ h $. Furthermore, criticality is equivalent to existence of a
\emph{ null-sequence}, i.e. a sequence $ (e_{n}) $ in $ C_{c}(X) $ such that $0\leq e_{n} \le v$,
 $ e_{n}\to v $ pointwise and $ h(e_{n})\to 0 $ for $ n\to\infty $. (Here $ v $ is $H$-superharmonic which is in fact, the ground state).

%%%%%%%%%%%%%%%%%%%%%%%%%%
\begin{defi}[Null-critical/positive-critical] Let $h$ be a quadratic form associated with the formal
Schr\"odinger operator ${\H}$, such that $h\ge0$ on $C_{c}(X)$. The
form $h$ is called \emph{null-critical}  (resp.,
\emph{positive-critical}) in $X$ with respect to a positive
potential $w$ if $h$ is critical  in $X$ and $\sum_X \psi^2 w
=\infty$ (resp., $\sum_X \psi^2 w <\infty$), where $\psi$ is the
ground state of $h$ in $X$.
\end{defi}
Note that the null/positive-criticality of a critical form
depends also on the weight $w$. By \cite[Theorem~6.2]{KePiPo1}, the form $ h-w $ is null-critical with respect to $ w $ if and only if the ground state $ \psi $ of the critical form  $h-w$ cannot be approximated by compactly supported functions, i.e., by convergence with respect to $ h-w $ and pointwise convergence.

Finally, we define the optimality criterion for Hardy weights we are interested in this paper.
\begin{defi}[Optimal Hardy-weight] We say that a positive function $w:X\to[0,\infty)$ is an
\emph{optimal Hardy-weight} for $h$ in $X$ if
\begin{itemize}
  \item $h-w$ is critical in $X$,
  \item $h-w$ is null-critical with respect to $w$ in $X$,
  \item  $h-w\ge\lm w$ fails to hold on $C_{c}(X\setminus W)$ for all
  $\lm>0$ and all finite $W\subseteq X$. In this case, we say that $w$ is \emph{optimal near infinity for $h$}.
\end{itemize}
\end{defi}
\subsection{The main results}
The following theorem is the main result of our paper.

\begin{thm}\label{thm_optimal_H} Let $b$ be a  connected graph over $X$,
and let $q$ be a given potential.  Let $u$ and $v$ be positive
$\H$-superharmonic functions that are $\H$-harmonic outside of a
finite set. Let $u_{0}:=u/v$, and assume that 
\begin{itemize}
  \item [(a)] $u_0:X\to (0,\infty)$ is proper.
\item  [(b)] $\displaystyle{\sup_{\substack{x,y\in X \\ x\sim y}}\frac{u_0(x)}{u_0(y)}}<\infty$.
\end{itemize}
Then the function
\begin{align*}
    w:=\frac{{\H}\left[(uv)^{1/2}\right]}{(uv)^{1/2}}
\end{align*}
is an optimal Hardy-weight in $X$, and
\begin{align*}
    w(x)=
   \frac{1}{2}\sum_{y\in X}b(x,y)
    \left[ \left(\frac{u(y)}{u(x)}\right)^{1/2}
    -\left(\frac{v(y)}{v(x)}\right)^{1/2}\right]^{2}
\end{align*}
for all $x\in X$ satisfying $\H u(x)=\H v(x)=0$.
\end{thm}

\begin{rem} Let us discuss  assumptions (a) and (b)  on the
function $u_{0}$:

The properness assumption (a) says that $u_{0}^{-1}(I)$ is a finite
set for any compact $I\subseteq u_{0}(X)$. Since $ u_{0}(X)
\subseteq (0,\infty) $ for positive superharmonic functions by the
Harnack inequality, \cite[Lemma~4.5]{KePiPo1}, the assumption
implies that $0$ and $\infty$ are the only (possible) accumulation
points of the range of $u_{0}$  (and without loss of generality we
may assume $\sup u_{0}=\infty$, since otherwise, we can only replace
$u_0$ with $\widetilde{u}_0:=1/u_0$). For example this can be
achieved if the limit of $u_{0}$ in the one-point compactification
$X\cup \{\infty\}$ of $X$ towards $\infty$ is  $\sup u_{0}=\infty$.

The second assumption bounds the quotient of  the function $u_0$ on
neighboring vertices. It can be understood as an anti-oscillation
assumption guaranteeing that the values of $u_{0}$ at neighbors do not
oscillate over the whole range of $u_0$. As above, it is also
satisfied if the limit of $u_{0}$ towards $\infty$ (in the one-point
compactification of $X$) exists and is either $0$ or $\infty$.
Another case when this is automatically satisfied is if the weighted
degree function $x\mapsto \sum_{y}b(x,y)$ is bounded on $X$ and $ \inf_{x\sim y}b(x,y)>0 $.
\end{rem}

Theorem~\ref{thm_optimal_H} uses the so called \emph{supersolution
construction} (see \cite{DFP})  with
the function $(uv)^{{1}/{2}}$ in the case where $u_0$, the quotient of the
supersolution $u$ and $v$, is a proper map. The corollary below applies for the case  when the quotient is
bounded. It uses the a similar supersolution construction that was
developed in \cite{DP} for the $p$-Laplacian.

\begin{cor}\label{thm_optimal_H_bounded}
Let $b$ be a  connected graph over $X$, and let $q$ be a given potential. Let
Let $u$ and $v$ be positive functions such  that $u$ and $v-u$ are
positive $H$-superharmonic functions on $X$ that are $\H$-harmonic outside of
a finite set. Let $u_{0}:=u/v$ and assume
\begin{itemize}
  \item  [(a)] The map $ u_{0}:X\to (0,1)$ is proper and satisfies $\sup u_{0}=1$.

  \medskip

\item  [(b)] $\displaystyle
{\sup_{\substack{x,y\in X \\ x\sim y}}
\frac{u_0(x)(1-u_{0}(y))}{u_0(y)(1-u_{0}(x))}}<\infty$.
\end{itemize}
Then the function
\begin{align*}
  w:=\frac{{\H}\left[u^{1/2}(v-u)^{1/2}\right]}
    {u^{1/2}(v-u)^{1/2}}
\end{align*}
is an optimal Hardy-weight in $X$, and
\begin{align*}
    w(x)=
   \frac{1}{2}\sum_{y\in X}b(x,y)
    \left[ \left(\frac{u(y)}{u(x)}\right)^{1/2}
    -\left(\frac{(v-u)(y)}{(v-u)(x)}\right)^{1/2}\right]^{2}
\end{align*}
for all $x\in X$ satisfying  $\H u(x)=\H v(x)=0$.
\end{cor}

%%%%%%%%%
\section{The toolbox}\label{s:toolbox}
%%%%%%%%%%
In this section we discuss the three major tools needed for the proof of the main theorem. The first is a discrete chain rule for the square root which is the basis for the supersolution construction of the Hardy weight. Secondly, we briefly recall the ground state transform which allows us to deal with Schr\"odinger operators instead of Laplacians only. Finally, we prove a co-area formula for Laplace type operators.

Throughout the section we are given a graph $b$ over a
discrete set $X$ and a  potential $q$  such that the associated form
satisfies $h\ge0$ on $C_{c}(X)$. Let $H=L+q$ be the corresponding
Schr\"odinger operator on $F(X)$ with the graph Laplacian $L=L_{b}$.
\subsection{Product and chain rules}
We present the product rule for the discrete Laplacian and a
discrete version of the chain rule for the square root.
%%%%%%%%%%%%%%
\begin{lemma}[Product rule]
Let $f,g\in {\F}(X)$. Then for all $ x\in X $,
\begin{equation*}
    {\H}(f g)(x) =(f{\H}g)(x)+(g{\L}f)(x)
    -\sum_{y\in X}b(x,y)(f(x)-f(y))(g(x)-g(y)).
\end{equation*}
\end{lemma}
\begin{proof}
This is a straightforward calculation.
\end{proof}
In general there is no chain rule for the discrete setting. However,
for the square root it was noticed by \cite{BHLLMY}
 that a chain
rule holds. This observation is a crucial point for the analysis in
this paper. Specifically, we use it to show that the square root of
a product of positive superharmonic functions is superharmonic.
\begin{lemma}[Chain rule for the square root]\label{l:squareroot}
Let $f,g\in {\F}(X)$ be  positive functions. Then for all $ x\in X $,
\begin{align*}
        2(f g)^{\frac{1}{2}}{\H}
        &\left[(f g)^{\frac{1}{2}}\right](x)=\left(f{\H}g\right)(x)
+\left(g{\H}f\right)(x)\\
    &+\!\sum_{y\in X}\!b(x,y)\!
    \left[g(x)^{\frac{1}{2}}
    \left({{f(x)^{\frac{1}{2}}}\!-\!{f(y)}^{\frac{1}{2}}}\right)\!-\!
f(x)^{\frac{1}{2}}
    \left({{g(x)^{\frac{1}{2}}}\!-\!{g(y)}^{\frac{1}{2}}}\right)
    \right]^{2}\!\!.
\end{align*}
\end{lemma}
\begin{proof}
To shorten notation we write $\nabla f=(f(x)-f(y))$. Then,
\begin{align*}
2({f}^{1/2}{\L}f^{1/2})(x)&
    \!=\!\!\sum_{y\in X}\! b(x,y)\!\left(\!{{f(x)}-f(y)+f(x)-2(f(x){f(y)})^{1/2}+f(y)}\!\right)\\
    &={\L f(x)}
    +\sum_{y\in
    X}b(x,y)\left(\nabla f^{1/2}\right)^{2}.
\end{align*}
We use the product rule
\begin{multline*}
2(f g)^{\frac{1}{2}}{\H}(f g)^{\frac{1}{2}}(x)
=2(f g)^{\frac{1}{2}}\left ({\L}(f g)^{\frac{1}{2}}(x)+q(f g)^{\frac{1}{2}}(x)\right)\\
=\!2(fg^{\frac{1}{2}}{\L}g^{\frac{1}{2}})(x)\!+\!2(gf^{\frac{1}{2}
}{\L}f^{\frac{1}{2}})(x) \!+\! 2(q f g)(x) \!-\! 2(f g)^{\frac{1}{2}}\!\!\sum_{x,y\in X}
b(x,y)\nabla f^{\frac{1}{2}}\nabla
g^{\frac{1}{2}}.
\end{multline*}
By the equality above we obtain
\begin{align*}
\ldots=&\left(f{\H}g\right)(x)
+\left(g{\H}f\right)(x)\\
&+\sum_{y\in X}
b(x,y)\left(f(\nabla
g^{\frac{1}{2}})^{2} +
g(\nabla
f^{\frac{1}{2}})^{2}-2(fg)^{\frac{1}{2}}\nabla f^{\frac{1}{2}}\nabla
g^{\frac{1}{2}}\right).\qedhere
\end{align*}
\end{proof}
\begin{cor}\label{c:squareroot}
 If $f,g\!\in \!\F(X)$ are positive $\H$-superharmonic (resp., $H$-harmonic) functions, then the function
$(f g)^{{1}/{2}}$ is $\H$-superharmonic (resp., $\H$-harmonic), and
\begin{align*}
\frac{\H[(fg)^{1/2}]}{(fg)^{1/2}}(x) =\frac{1}{2}\sum_{y\in X}b(x,y)\left[
    \left(\frac{f(y)}{f(x)}\right)^{1/2}-
    \left(\frac{{g(y)}}{g(x)}\right)^{1/2}
\right]^{2}
\end{align*}
for all $x$ satisfying  $\H f(x)=\H g(x)=0$.

In particular, if
$f$ is positive $L$-superharmonic (resp., $L$-harmonic), then $f^{{1}/{2}}$ is
$L$-superharmonic (resp., $L$-harmonic),{ and
\begin{align*}
\frac{\L[f^{1/2}]}{f^{1/2}}(x)=\frac{1}{2f(x)}\sum_{y\in X}b(x,y)
    \left(f(x)^{1/2}-f(y)^{1/2}\right)^{2}
\end{align*}
 for all $x$ satisfying  $\L f(x)=0$.}
\end{cor}
\begin{proof}
We calculate $H[(f g)^{{1}/{2}}]$ using the chain rule for the square root (Lemma~\ref{l:squareroot}) to obtain the super-harmonicity and the explicit expression for $w$ at points
where $f$ and $g$ are $\H$-harmonic. Indeed, $\frac{\H[(fg)^{1/2}]}{(fg)^{1/2}} (x) $
\begin{align*}
= \sum_{y\in X}&\frac{b(x,y)}{2(fg)(x)}
    \left[g(x)^{\frac{1}{2}}
    \left({{f(x)^{\frac{1}{2}}}\!-\!{f(y)}^{\frac{1}{2}}}\right)\!-\!
f(x)^{\frac{1}{2}}
    \left({{g(x)^{\frac{1}{2}}}\!-\!{g(y)}^{\frac{1}{2}}}\right)
    \right]^{2}
\\
&=\frac{1}{2}\sum_{y\in X}{b(x,y)}
    \!\!\left[{f(x)}^{-\frac{1}{2}}
    \left({{f(x)^{\frac{1}{2}}}\!-\!{f(y)}^{\frac{1}{2}}}\right)\!-\!
{g(x)}^{-\frac{1}{2}}
    \left({{g(x)^{\frac{1}{2}}}\!-\!{g(y)}^{\frac{1}{2}}}\right)
    \right]^{2}\\
&=\frac{1}{2}\sum_{y\in X}b(x,y)\left[
    \left(\frac{f(y)}{f(x)}\right)^{\frac{1}{2}}-
    \left(\frac{g(y)}{g(x)}\right)^{\frac{1}{2}}
\right]^{2}.
\end{align*}
The statement for $L$ follows from the calculation above with $g=1$.
\end{proof}
%%%%%%%%%%%%%%%%%%%%%%%%%%

\subsection{The ground state transform}
The ground state transform is vital to deal with Schr\"odinger operators $ H=L+q $ specifically under the presences of  potentials $ q $ with non-vanishing negative part $ q_{-}\neq0 $. Namely, given a strictly positive (super)harmonic function one can reduce the analysis to the one of Laplace type operators (resp. Schr\"odinger operators with positive potentials).

This transform was first used to show Hardy inequalities by \cite{FSW08} and later by \cite{BG}. We present here the notation that is needed for the paper and refer the reader for a more detailed discussion to \cite[Section~4.2]{KePiPo1}.

For any function $ v\in C(X) $ we denote the operator of multiplication by $ v $ by $ T_{v} $ and note that whenever $ v $ does not vanish its inverse is given by $ T_{v^{-1}} $. For strictly positive $ v $,  we define ground state transform of a Schr\"odinger operator $ \H $  on $ T_{v}^{-1} F(X)$
\begin{align*}
\H_{v}=T_{v}^{-1}\H T_{v}
\end{align*}
and if additionally $ v\in \F(X)$, we define the bilinear form $h_{v}: C_{c}(X)\times
C_{c}(X)\to \mathbb{R}$ via
\begin{align*}
h_{v}(\ph,\psi):=\frac{1}{2}\sum_{x,y\in X}
b(x,y)v(x)v(y)(\ph(x)-\ph(y))(\psi(x)-\psi(y)).
\end{align*}

These two notions are related by the following formula.
See \cite[Proposition~4.8]{KePiPo1} for a proof and references therein for a discussion of the history in the discrete setting.
\begin{pro} [Ground state transform]\label{t:GST}	Let $v\in \F(X)$ be strictly positive, $f\in C(X)$ such that $\H v=f
	v$. Then
		 \begin{align*}
	 h(\ph,\psi)
	 =h_{v}\left({\frac{\ph}{v},\frac{\psi}{v}}\right)
	 +\langle f\ph,\psi\rangle \qquad \forall \ph,\psi\in C_{c}(X).
	 \end{align*}
\end{pro}

Clearly, whenever there exists a strictly positive $\H$-harmonic
function $ v$, criticality can be carried over directly from $h$ to
$h_{v}$ and vice versa.  In this paper we are  concerned with
criticality of a form $h-w$ for some positive function $ w\ge0 $. However,
criticality, null-criticality and even optimality near infinity can
be carried over from $h_{v}-v^{2}w$ to $h-w$ as well.

\begin{cor}\label{c:gst} Let  $v\in \F(X)$ be a strictly positive
	$\H$-(super)harmonic function and $w\ge0$.
	\begin{itemize}
		\item [(a)] $h-w$ is critical if and only if $h_{v}-v^{2}w$ is
		critical.
		\item [(b)]
		$h-w$ is null-critical with respect to $w$ if and only if $h_{v}-v^{2}w$ is
		null-critical with respect to $v^{2}w$.
		\item [(c)] $w$ is  optimal near infinity for $h$ if and only if $v^{2}w$
		is  optimal near infinity for $h_{v}$.
	\end{itemize}
\end{cor}
\begin{proof} By the ground state transform the quadratic form associated to the operator $ \H_{v}-w $  is $ h_{v}-v^{2}w $ (where Proposition~\ref{t:GST} is applied for the operator $ \H-w $ and $ f=w $). Furthermore, for any $\psi\in \F(X)$  we see that
	\begin{align*}%\label{h-w}
	(\H-w)\psi =     (T_{v}\H_{v} T_{v^{-1}}-w)\psi
	=   T_{v}(\H_{v}-w)T_{v^{-1}}\psi = T_{v} (\H_v-w) \frac{\psi}{v}\,.
	\end{align*}
	Therefore, every positive $(\H-w)$-(super)harmonic function $\psi$ yields a
	positive $(\H_{v}-w)$-(super)harmonic function $\psi/v$. Thus, $h-w$
	is critical if and only if $h_v-v^{2}w$ is critical by \cite[Theorem~5.3]{KePiPo1}. This proves (a). For the very same reason,
	statement (b) follows immediately by the definition of  null-criticality.
	
	%of $h-w$ with respect to $w$ means that $\psi$ can approximated by a sequence
	%$(\ph_{n})$ in  $C_{c}(X)$ with respect to $\|\cdot\|_{h-w,o}$. By
	%the ground state transform $(\ph_{n}/\psi)$ approximates $1$
	%with respect to $\|\cdot\|_{(h-w)_{\psi},o}$ and vice versa.
	
	Finally, $w$ not being optimal near infinity for $h$ means there is a
	finite $W\subseteq X$ and $\lm>0$ such that $h-w\ge\lm w$ on $C_{c}(X\setminus W)$. Hence,
	$h-(1+\lm/2)w$ is subcritical in $X\setminus W$ which according to
	(a) yields that $h_{ v}-(1+\lm/2) v^{2}w$ is subcritical there. Consequently,
	$h_{v}-v^{2}w\ge (\lm/2) v^{2}w$ on $C_{c}(X\setminus W)$ and thus, $v^{2}w$ is
	not optimal near infinity (for $h_{v}$). By the same argument, $w$ is not optimal
	near infinity for $h$ when $v^{2}w$ is not optimal near infinity for $h_{v}$ which
	shows (c).
\end{proof}

%%%%%%%%%%%%%%%%%%%%%%%%%%
\subsection{Coarea formula}\label{s:coarea}
%%%%%%%%%%%%%%%%%%%%%%%%%%
In this section we establish the pivotal tool for the proof of the main
theorems. It allows us to translate calculations and estimates of
infinite sums over graphs to one dimensional integrals.
%%%%%%%%%%
\begin{thm}\label{t:coarea}
Let $b$ be a  connected graph over $X$, and let $u\in C(X)$ be positive. Let  $f:(\inf u,\sup u)\to[0,\infty)$ be a
Riemann integrable function. Then
\begin{align}\label{eq_coarea}
\frac{1}{2}\sum_{x,y\in X\times
X}b(x,y)&(u(x)-u(y))\int^{u(x)}_{u(y)}f(t)\,\mathrm{d}t = \int_{\inf u}^{\sup
u}f(t)g(t)\,\mathrm{d}t,
\end{align}
 where both sides can take the value
$+\infty$, and  $g:(\inf u,\sup u)\to[0,\infty]$ is given by
\begin{align*}
g(t):=    \sum_{\substack{x,y\in X\\u(y)< t\leq u(x)}}b(x,y)(u(x)-u(y)).
\end{align*}
Assume further that $u\in F(X)$ is $L$-harmonic outside of a finite
set,  and
\begin{itemize}
  \item[(a)] $u^{-1}(I)$ is finite for any compact $I\subseteq (\inf u,\sup u)$,
  \smallskip
  \item[(b)] $\displaystyle{\sup_{\substack{x,y\in X \\ x\sim y}}\frac{u(x)}{u(y)}<\sup_{x,y\in X }\frac{u(x)}{u(y)}}$.
\end{itemize}
Then there are positive constants $c$ and $C$ such that
\begin{align*}
    c\leq g\leq C,
\end{align*}
 and if in addition $u$ is $L$-harmonic in $X$, then $g$ is constant.
\end{thm}
\begin{rem}(a) Note that by  $f\ge0$, the terms in the sum
on both sides  of the equality \eqref{eq_coarea} above are always
 greater than or equal to zero.

(b) Let $u$ and $f$ be as in Theorem~\ref{t:coarea}  with  $f$ being
continuous. Then by the Lagrange's mean value theorem, if $u(x)\neq u(y)$ there is
$\theta_{x,y}\in (u(x)\wedge u(y),u(x)\vee u(y))$ such that
\begin{align*}
    f(\theta_{x,y})
    =\frac{\int^{u(x)}_{u(y)}f(t)\,\mathrm{d}t}{u(x)-u(y)}\,.
\end{align*}
Consequently, the coarea formula reads as
\begin{align*}
\frac{1}{2}\sum_{x,y\in X\times
X}b(x,y)(u(x)-u(y))^{2}f(\theta_{x,y})= \int_{\inf u}^{\sup
u}f(t)g(t)\,\mathrm{d}t.
\end{align*}
\end{rem}

The following lemma can be interpreted as a Stokes type theorem.
Specifically, the function $g$ can be viewed as the integral of the
normal derivative over the boundary of the level set $\{x\mid
u(x)>t\}$ of $u$ for some $t$. Formula \eqref{eq_gt2} of the Lemma~\ref{l:co-area} then shows that the
function $g$ at $t_{1}$ and $t_{2}$ differs only by the nonharmonic
contribution of $u$ on the set $A=\{x\mid t_{1}<u(x)\leq t_{2}\}$.
Provided \eqref{eq_gt2}, the proof of the coarea-formula reduces to
algebraic manipulations and the application of Tonelli's theorem.

\begin{lemma}[Stokes-type formula]\label{l:co-area} Let $u\in \F(X)$
be a positive nonconstant function such that $u^{-1}(I)$ is finite for any compact  $I\subseteq (\inf u,\sup u)$.
Let
\begin{equation*}
g:(\inf u,\sup u)\to[0,\infty], \qquad g(t):=\!\!\!\sum_{\substack{x,y\in X\\u(y)< t\leq u(x)}}\!\!b(x,y)(u(x)-u(y)).
\end{equation*}
Then for any $t_{1},t_{2}\in(\inf u,\sup u)$  such that $t_{ 1}\leq
t_{2}$, the set $$A:=\{x\in X\mid t_{1}< u(x)\leq t_{2}\}$$ is finite,
and
\begin{align}\label{eq_gt2}
    g(t_{2})=g(t_{1})-\sum_{x\in A}\L u(x),
\end{align}
where both sides may take the value $+\infty$. In particular, $g$ is
monotone decreasing whenever $u$ is $L$-superharmonic.

Moreover, if
$$\sup_{x,y\in X \, x\sim y}\frac{u(x)}{u(y)}<\sup_{x,y\in X }\frac{u(x)}{u(y)}\,,$$
and
$u$ is $L$-harmonic outside of a finite set, then $g$ is piecewise
constant with finitely many jumps, and for some positive constants $c,C$
\begin{align*}
0<    c\leq g\leq C<\infty.
\end{align*}
Furthermore,  if $u$ is $L$-harmonic in $X$, then $g$ is constant.
\end{lemma}
\begin{proof}
For this proof we denote $\nabla f:=f(x)-f(y)$ for functions $f$
whenever summing over $x$ and $y$ such that $x\sim y$.

For $t>0$, define $$\Omega_{t}:=\{x\in X\mid u(x)>t\}.$$ Let
$t_{1},t_{2}\in (\inf u,\sup u)$ with $t_{1}\le t_{2}$, and define
$A$ to be
$$A:=\Omega_{t_{1}}\setminus\Omega_{t_{2}}=\{x\in X\mid t_{1}<u(x)\leq
t_{2} \}. $$

By the assumption that the pre-images of $u$ of compact sets in the interval
$(\inf u,\sup u)$ are finite, the set $A$ is finite. Therefore, the
characteristic function $1_{A}$ of $A$ is in $C_{c}(X)$. For $B\subset X$ we denote
$$\partial B:=\{(x,y)\in X\times X \mid x\in B, \; y\not\in B\}.$$

Since $u\in
\F(X)$,  we can apply the Green formula, \cite[Lemma 4.7]{HK}, for $u$ paired with $1_{A}$ to
see
\begin{align*}
\sum_{A}{\L}u     &=\sum_{ X}1_{A}{\L}u
    =\frac{1}{2}\sum_{X\times
    X}b\nabla u\nabla 1_{A}=\sum_{ \partial A
    } b\nabla u,
\end{align*}
where the right hand side also converges absolutely.

In the next step we show that the sum on the left hand side can be
split into a difference of a sum over the boundary of $\Om_{t_1}$
and $\Om_{t_{2}}$. To this end, we observe that for any $B\subseteq X$, we have $(x,y)\in \partial
B$ if and only if $(y,x)\in
\partial (X\setminus B)$.
Moreover, since $\Omega_{t_{2}}\subseteq \Omega_{t_{1}}$, we
conclude
\begin{align*}
\partial A=\left(\partial \Omega_{t_{1}}\cap\partial A\right)
\cup \left(\partial (X\setminus \Omega_{t_{2}})\cap \partial
A\right),
\end{align*}
and
\begin{align*}
\partial\Omega_{t_{1}}\setminus \partial A
=\partial\Omega_{t_{2}}\setminus \partial (X\setminus A).
\end{align*}
%%%%%%%%%%%%%%%%
%\Hmm{I temporarily removed the figure, somehow it does not work on my computer}
%\begin{comment}
\begin{figure}[h]
\begin{overpic}
[width=1\textwidth]{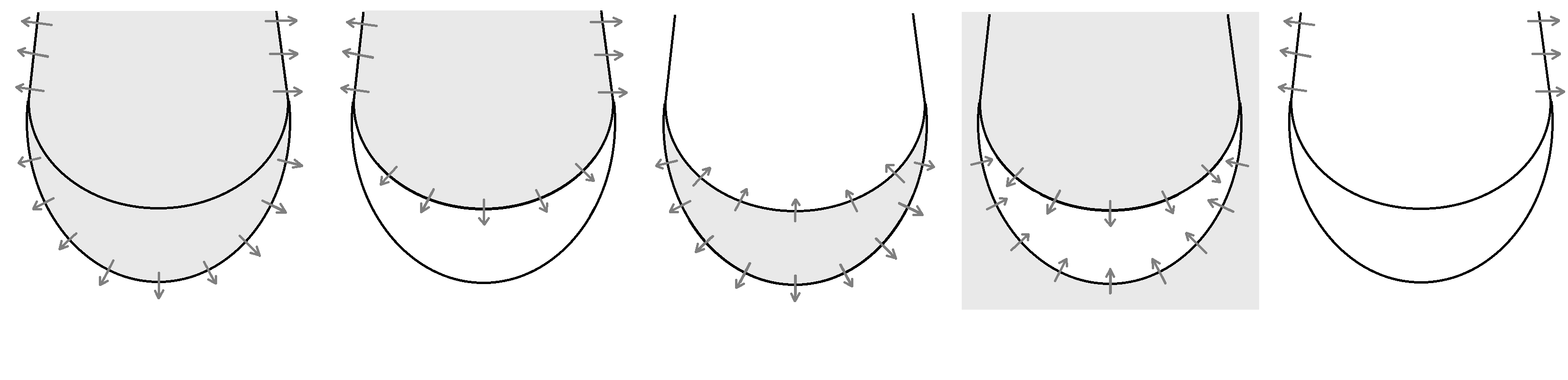}
\put(30,26){$\Omega_{t_{1}}$}\put(23,1){{\large$\partial\Omega_{t_{1}}$}}
\put(105,50){$\Omega_{t_{2}}$}\put(100,1){{\large$\partial\Omega_{t_{2}}$}}
\put(180,26){$A$}\put(172,2){{\large$\partial{A}$}}
\put(240,50){$X\setminus{A}$}\put(232,1){{\large$\partial({X\!\setminus\!{A}})$}}
\put(305,7){{$\partial\Omega_{t_{1}}\!\setminus\! \partial
A$}}\put(290,-7){{$ =\partial\Omega_{t_{2}}\!\!\setminus\!\!
\partial (X\!\!\setminus\!\! A)$}}
\end{overpic}
\caption{Illustrations of the sets and their boundaries}
\end{figure}
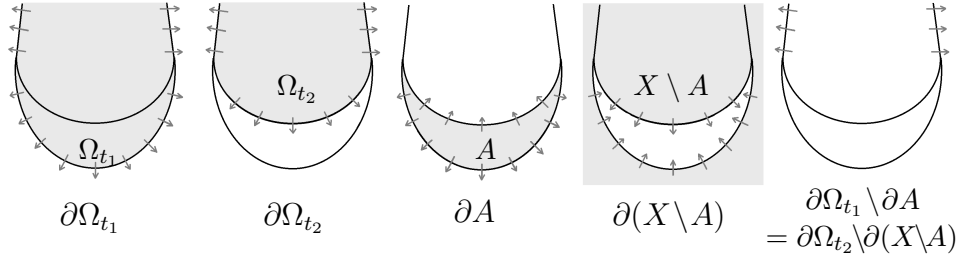
%\end{comment}
%%%%%%%%%%%%%%%%%%%%%%%%%%%%%%%%%%%%%%%%%
Thus, since $\sum_{(x,y)\in\partial A}b(x,y)(u(x)
-u(y))$ converges absolutely, we obtain by the considerations above
\begin{align*}
\sum_{\partial A}b\nabla u
 &=\sum_{\partial \Omega_{t_{1}}\cap\partial A}\hspace{-.2cm}b\nabla u
+ \sum_{\partial (X\setminus \Omega_{t_{2}})\cap \partial A}
\hspace{-.3cm}b\nabla u =\sum_{\partial \Omega_{t_{1}}\cap\partial
A}\hspace{-.2cm}b\nabla u -\sum_{\partial\Omega_{t_{2}}\cap \partial
(X\setminus A)}\hspace{-.3cm}b\nabla u.
\end{align*}
We employ the equalities above to see
\begin{align*}
    g(t_{2})&=\sum_{\partial \Omega_{t_{2}}}b\nabla u
    =\sum_{\partial \Omega_{t_{2}}\setminus
    \partial (X\setminus A)}b\nabla u +
    \sum_{\partial \Omega_{t_{2}}\cap \partial (X\setminus
    A)}b\nabla u\\
    &=
    \sum_{\partial \Omega_{t_{1}}\setminus\partial
    A}b\nabla u+\sum_{\partial \Omega_{t_{2}}\cap\partial(X\setminus A)}
b\nabla u\\
    &= g(t_{1})-
    \sum_{\partial \Omega_{t_{1}}\cap \partial A}b\nabla u+
    \sum_{\partial \Omega_{t_{2}}\cap\partial(X\setminus A)}b\nabla u
\\
    &= g(t_{1})-
    \sum_{\partial \Omega_{t_{1}}\cap \partial A}b\nabla u
    - \sum_{\partial(X\setminus \Omega_{t_{2}})\cap
\partial A}b\nabla u\\
    &= g(t_{1})-\sum_{\partial A}b\nabla u
     =g(t_{1})-\sum_{A}\L u.
\end{align*}
As $\sum_{A}\L u <\infty$, this shows for all $t_{1},t_{2}\in (\inf u,\sup u)$ we have $g(t_{1})<\infty$  if and only
if $g(t_{2})<\infty$.

If $u$ is $L$-harmonic outside of a finite set, then
 the mapping $ \{B\subseteq X\}\to \R $, $B\mapsto \sum_{B}\L u$ takes only finitely many
values. Therefore, by \eqref{eq_gt2}, $g$ takes only finitely many
values. Moreover, since $\sum_{A}\L u$ as a function of $t_{2}$
changes only at finitely many $t_{2}$ and vice versa for $t_{1}$, the function
$g$ is piecewise constant with finitely many jumps.

Hence, $g\ge c>0$ can fail only if $g(t)=0$ for some $t\in(\inf
u,\sup u)$. However, this is impossible as $g(t)=0$ implies
$\partial \Om_{t}=\emptyset$ which implies either $t<\inf u$ or
$t\ge \sup u$ by the connectedness of the graph.

To see the upper bound for $g$, we employ the assumption $$\sup_{x,y\in X \,
x\sim y}\frac{u(x)}{u(y)}<\sup_{x,y\in X}\frac{u(x)}{u(y)}\,.$$ This
assumption is equivalent to
$$c:=\sup_{x\sim y}(u(x)-u(y))<\sup_{x\in X} u(x) - \inf_{x\in X} u(x).$$
 Hence, there are $t_{1},t_{2}\in (\inf u,\sup u)$ such that
$t_{2}-t_{1}> c$. For the choice of these $t_{1},t_{2}$, there is no
vertex in $\Om_{t_{2}}$ that is connected to a vertex outside of
$\Om_{t_{1}}$. Hence, $\partial \Omega_{t_{2}}=
\partial \Omega_{t_{2}}\cap\partial (X\setminus A)$ and we have by
the considerations above
\begin{align*}
g(t_{2})= \sum_{\partial \Omega_{t_{2}}\cap\partial (X\setminus
A)}\hspace{-.3cm}b\nabla u \leq \sum_{\partial A}b|\nabla u|<\infty.
\end{align*}
Thus, $g$ stays finite on $(\inf u,\sup u)$ and since $g$ is
piecewise constant with finitely many jumps, there is $C$ such that
$g\leq C$. This finishes the proof.
\end{proof}

\begin{proof}[Proof of Theorem~\ref{t:coarea}]
Let  $t>0$, and recall that
$$\Omega_{t}=\{x\in X\mid u(x)>t\}.$$ Let $1_{x,y}$
be the characteristic function of the interval
$$I_{x,y}=[u(x)\wedge u(y),u(x)\vee u(y)].$$ Observe that $(x,y)$ or $(y,x)$ are in
$\partial \Omega_{t}:=\Omega_{t}\times X\setminus \Omega_{t}$ if and
only if
$t\in I_{x,y}$. %To shorten notation we write $\nabla u=(u(x)-u(y))$.
With this observation in mind, we calculate
\begin{align*}
\sum_{x,y\in
X\times X} b(x,y)&(u(x)-u(y))\int^{u(x)}_{u(y)}f(t)\,\mathrm{d}t\\
&= \sum_{x,y\in  X\times X}b(x,y)|u(x)-u(y)|\int_{\inf u}^{\sup
u}f(t)1_{x,y}(t)\,\mathrm{d}t.
\end{align*}
Now, by Tonelli's theorem we obtain
\begin{align*}
\ldots&=\int_{\inf u}^{\sup u}f(t)\sum_{x,y\in X\times
X}b(x,y)|u(x)-u(y)|1_{x,y}(t)\,\mathrm{d}t\\
&=2\int_{\inf u}^{\sup u}f(t)\sum_{(x,y)\in
\partial \Omega_{t}} b(x,y)|u(x)-u(y)|\,\mathrm{d}t\\
&=2\int_{\inf u}^{\sup u}f(t)\sum_{(x,y)\in \partial \Omega_{t}}
b(x,y)(u(x)-u(y))\,\mathrm{d}t
\end{align*}
since $u(x)\ge u(y)$ for $(x,y)\in\partial \Omega_{t}$. This shows
the first part of the theorem. The second part follows from
Lemma~\ref{l:co-area}.
\end{proof}
%%%%%%%%%%%%%%%%%%%%%%
\section{Critical Hardy-weights}\label{s:crit}
%%%%%%%%%%%%%

In the following three sections we will prove
Theorem~\ref{thm_optimal_H} for operators $H=L+q$ with finitely supported  $q\ge0$.
This will be
achieved by the virtue of the coarea formula (Theorem~\ref{t:coarea}). The general case
will  be deduced in Section~\ref{sec_pf_main} using the ground state transform. We start by proving criticality in this
section, and show null-criticality and optimality at infinity in the
two succeeding sections.

\begin{thm}\label{t:critical}Let $b$ be a connected graph and $q\ge0$ be a
finitely supported potential. Suppose that $u$ is a
positive $\H$-superharmonic function that is $\H$-harmonic outside
of a finite set, and satisfies
\begin{itemize}
  \item [(a)] $u:X\to (0,\infty)$ is proper,
\medskip
  \item[(b)]
  $\displaystyle{\sup_{\substack{x,y\in X \\ x\sim
  y}}\frac{u(x)}{u(y)}}<\infty$.
%$  <\displaystyle{\sup_{x,y\in X\setminus K}\frac{u(x)}{u(y)}}$.
  \end{itemize}
Let $h$ be the quadratic form associated to $\H$. Then $h- w$ with
$w:=\frac{\H\big[ u^{{1}/{2}}\big]}{u^{{1}/{2}}}$, is critical in
$X$.
\end{thm}
\begin{proof}
We set $v:=u^{{1}/{2}}$. Then $v$ is positive and $\H$-superharmonic
by the chain rule for the square root (Lemma~\ref{l:squareroot})
since $q\ge0$. Furthermore,  $v$ is obviously a positive $(\H-w)$-harmonic
function in $X$.

The strategy of the proof is to
construct a null-sequence $ (e_{n})  $ in $ C_{c}(X)  $ with respect to $(h-w)_{v}$, i.e.
$(h-w)_{v}(e_{n})\to 0$ and $ e_{n}\to 1 $ pointwise . By
\cite[Theorem~5.3~(iv)]{KePiPo1}, this then implies that
$(h-w)_{v}$ is critical, and hence the criticality of $h-w$.

 Set $\ph_{n}:\R\to\R$
\begin{multline*}
    \ph_{n}(t):=
      \left(  {2 + \frac{1}{\log n}\log(t)}  \right)
    1_{[\frac{1}{n^{2}},+\frac{1}{n}]}(t)
     + 1_{[\frac{1}{n},n]}(t)\\
     + \left(  {2 - \frac{1}{\log n}\log (t)}  \right)  1_{[n,n^{2}]}(t)\,,
\end{multline*}
 and let $e_{n}:=\ph_{n}\circ u$. Since $\mathrm{supp}\, \ph_{n}\subseteq
(0,\infty)$, and $\sup u=\infty$ or $\inf u=0$, we have $e_{n}\in C_{c}(X)$ by
assumption (a). Obviously, $e_{n}\to 1$ pointwise as $n\to\infty$.
So, we are left to show $( h-w)_{v}(e_{n}) \to 0$ as $n\to\infty$.
We compute
\begin{align*}
(h-w)_{v}({e_{n}})&=\frac{1}{2}\sum_{x,y\in
X}b(x,y)(u(x)u(y))^{{1}/{2}}({\ph_{n}(u(x))-\ph_{n}(u(y))})^{2}\\
&=\frac{1}{2}\sum_{x,y\in
X}b(x,y)\left({{u(x)-u(y)}}\right)c(x,y)\left({{\int^{u(x)}_{u(y)}t\ph_{n}'(t)^{2}\,\mathrm{d}t}}\right),
\end{align*}
where
\begin{equation*}
    c(x,y) :=
    \frac{(u(x)u(y))^{{1}/{2}}(\ph_n(u(x))-\ph_n(u(y)))^{2}}
    {(u(x)-u(y)){\int^{u(x)}_{u(y)}t \ph_{n}'(t)^{2}\,\mathrm{d}t}}
\end{equation*}
whenever the denominator is nonzero and $c(x,y)=0$ otherwise.

Since $c(x,y)$ always appears in a product with $b(x,y)$ it suffices
to consider $x,y$ with $x\sim y$.
By the anti-oscillation assumption~(b) there is a constant $ C_{0}:=\sup_{z\sim w}u(z)/u(w) $ such that for $  n > \sqrt{ C_0} $ we  have  $ u(x),u( y) \in(0,n]  $ or $ u(x),u(y)\in[1/n,\infty) $ for $ x\sim y $.
 We now use  the definition of $\ph_{n}$ and the elementary inequalities
\begin{align*}
|a\wedge c-b\wedge c|&\leq |a-b|, \qquad a,b,c\in \R, \\
\frac{\log b-\log a}{b-a}&\leq\log'(a)=\frac{1}{a},\qquad 0<a\le b<\infty,
\end{align*}
to estimate
\begin{align*}
     c(x,y)&\leq    \frac{(u(x)u(y))^{{1}/{2}}(\log u(x)-\log u(y))}
    {(u(x)-u(y))}\\
    &\leq\sup_{z,w\in
    X,z\sim w}\left({\frac{u(z)}{u(w)}}\right)^{{1}/{2}}=C_{0}
\end{align*}
for all $x\sim y$ and  $  n > \sqrt{ C_0} $. Notice that $C_{0}<\infty$ by our assumption (b).

We use this estimate and we apply now the coarea formula
 with $f(t)=t\ph'_n(t)^{2}$. To this end we note
that the assumptions of Theorem~\ref{t:coarea} are fulfilled: The
function $u$ is $\L$-harmonic outside of the finite set (including
the finite set where $q$ is supported), Theorem~\ref{t:coarea}~(a) is
fulfilled by assumption~(a), and  Theorem~\ref{t:coarea}~(b) is
fulfilled by assumption (b) and $\sup u=\infty$ or $\inf u=0$. Moreover, by the $\L$-harmonicity
of $u$ outside of a finite set, the function $g$ in the coarea
formula \eqref{eq_coarea} is piecewise constant. Therefore, there
exists a constant $C_{1} $ such that
\begin{align*}
(h-w)_{v}(e_{n})&\leq C_{0}\sum_{x,y\in X}
b(x,y)\left({{u(x)-u(y)}}\right)
\left({{\int^{u(x)}_{u(y)}t\ph_{n}'(t)^{2}\,\mathrm{d}t}}\right)\\
&\leq C_{1}\int^{\sup u}_{\inf u}t  \ph_{n}'(t)^{2}\,\mathrm{d}t\\
&\leq C_{2}\left({\frac{1}{\log
n}}\right)^{2}\left({\int_{\frac{1}{n^{2}}}^{\frac{1}{n}}\frac{\,\mathrm{d}t}{t}}
+\int_{{n}}^{{n}^{2}}\frac{\,\mathrm{d}t}{t}\right) =\frac{2C_{2}}{\log n}\xrightarrow[n\to\infty]{} 0.
\end{align*}
Thus, $ (e_{n})  $ is a null-sequence which implies that $h-w$ is critical by the discussion in the beginning of the proof.
\end{proof}

%%%%%%%%%%%%%%%
\section{Null-criticality}\label{s:nullcrit}

In the present section we prove the null-criticality assertion of
Theorem~\ref{thm_optimal_H}
under the additional assumption that $q$ is a positive finitely supported potential. For general $q$ the
statement is then deduced in
Section~\ref{sec_pf_main} using the ground state transform.

In the case $q\ge0$, it is convenient to extend the quadratic form
$h$ for a graph $b$  which is defined on  $C_{c}(X)$ to a map $C(X)\to[0,\infty]$ via
\begin{align*}
    f\mapsto \frac{1}{2}\sum_{x,y\in X}b(x,y)(f(x)-f(y))^{2}+\sum_{x\in X}q(x)f(x)^{2}.
\end{align*}
It is easily  seen that this defines a quadratic form and with
a slight abuse of notation we denote this form also by $h$. Moreover, whenever there is a positive harmonic function $u$ for
the operator $\H$ associated to a form $h$ with general potential
$q$, such an extension can be employed for the ground state
transform $h_{u}$ of $h$.

%%%%%%%%%%%%%%%%%%%%%%%%%%%%%%
\begin{thm}\label{t:nullcrit}
Let $b$ be a graph, $q\ge0$ be finitely supported and $\H=\L+q$. Let
$u$ be a positive $\H$-superharmonic function that is $H$-harmonic outside of a finite set.
In addition, assume that
\begin{enumerate}
  \item[(a)] $u:X \to (0,\infty)$ is proper,

  \item[(b)] $\displaystyle{\sup_{\substack{x,y\in X \\ x\sim y}}
  \frac{u(x)}{u(y)}}<\infty$.
\end{enumerate}
Let $h$ be the quadratic form associated to $\H$. Then
 $$ h(u^{{1}/{2}})=\infty,$$
and the form $h-w$ with
$w:=\frac{{\H}[u^{1/2}]}{u^{{1}/{2}}}$ is null-critical in $X$ with respect to $w$.
\end{thm}
\begin{proof}
 By the elementary inequality
\begin{align*}
 \frac{(b^{{1}/{2}}-a^{{1}/{2}})^{2}}
 {(b-a){\int_{a}^{b}\frac{\,\mathrm{d}t}{t}}}
 \ge \frac{a}{4b}\,,\qquad 0< a< b<\infty,
\end{align*}
 we have for $x,y\in X$ with $x\sim y$
\begin{align*}
    c(x,y)
    :=\frac{(u^{{1}/{2}}(x)-
    u^{{1}/{2}}(y))^{2}}{(u(x)-u(y))\int_{u(y)}^{u(x)}\frac{\,\mathrm{d}t}{t}}
    \ge \frac{1}{4}\inf_{x'\sim y'}\frac{u(x')}{u(y')}=:C_{0},
\end{align*}
where $C_0>0$ by assumption (b) whenever $ u(x)\neq u(y) $. We apply the
coarea formula (Theorem~\ref{t:coarea}) with $f(t)=1/t$. Since $q$
is finitely supported, $u$ is $\L$-harmonic outside of
a finite set. Thus, by Theorem~\ref{t:coarea}, the function $g$ in
the coarea formula \eqref{eq_coarea}  is bounded away from zero, and we
get
\begin{multline*}
\frac{1}{2}\!\!\sum_{x,y\in X}\!\!b(x,y)(u^{{1}/{2}}(x)-
    u^{{1}/{2}}(y))^{2}
    \!=\!\frac{1}{2}\!\!\sum_{x,y\in X}\!\! b(x,y)c(x,y)(u(x)-u(y))\!\!\int_{u(y)}^{u(x)}\!\!
    \frac{\,\mathrm{d}t}{t}\\
    \ge C_{0}\sum_{x,y\in
    X}b(x,y)(u(x)-u(y))\int_{u(y)}^{u(x)}\frac{\,\mathrm{d}t}{t}
    \ge C_{1}\int^{\sup u}_{\inf u}\frac{1}{t}\,\mathrm{d}t
=\infty,
\end{multline*}
where $C_{1}$ is a positive constant. Consequently, $h(u^{{1}/{2}})=\infty$.

It remains to show that this implies the null-criticality of $h-w$ with respect to $w$
in the case $\sup u=\infty$ or $\inf u=0$. Note that the function $u^{1/2}$ is
$(H-w)$-harmonic with $w=\frac{{\H}[u^{1/2}]}{u^{{1}/{2}}}$. Since $\sup
u=\infty$  or $\inf u=0$, it follows from Theorem~\ref{t:critical}
 that, under our assumptions, $h-w$ is critical.

 In the critical case there is a unique positive harmonic function (up to linear dependence), see  \cite[Theorem~5.3~(iii)]{KePiPo1}. Hence, this ground state is
 $u^{1/2}$.  By $h(u^{{1}/{2}})=\infty$ it follows, in particular, that $ u^{1/2} $ can not be approximated by compactly supported functions. By a characterization of null-criticality \cite[Theorem~6.2]{KePiPo1}, this implies that the form $ h -w$ is null-critical with respect to $ w $.
\end{proof}
%%%%%%%%%%%%%%%%%%
%\begin{rem}\label{rem_3} The null-criticality assertion of Theorem~\ref{thm_optimal_H}  for the operator $H-w$ follows from Theorem~\ref{t:nullcrit} using the ground state transform (see Section~\ref{sec_pf_main}).\end{rem}
%%%%%%%%%%%%%%%%%

%%%%%%%%%%%%%%%%%%%%%%%%
\section{Optimality near infinity}\label{s:optimal}
%%%%%%%%%%%%%%%%%%%%%%%%%%%%%%%%%%%%%
In this section we give a criterion for optimality near infinity.
Recall that  if $h$, $u$ and $ q\ge0 $ satisfy the assumptions of Theorem~\ref{t:nullcrit}, then
$$h(u^{1/2})= \sum_{x,y\in X}b(x,y)\left(u(x)^{1/2}-u(y)^{1/2}\right)^{2}
+\sum_{x\in X}q(x)u(x)^{1/2}=\infty,$$ where  $h$ denotes
again the extension of the form on $C_{c}(X)$ to $C(X)$. We will
deduce optimality at infinity  for $q\ge0$ directly from the
divergence of the sum above. The case of general $q$ is then covered
in Section~\ref{sec_pf_main}.

\begin{thm}\label{t:optimal}
Consider a graph  $b$ and $q\ge0$ with a finite support. Let $\H:=\L+q$. Let $u$ be a
positive nonconstant $\H$-superharmonic function in $X$ that is
$L$-harmonic outside of a finite set.  Furthermore,
let
 $w=L(u^{{1}/{2}})/u^{{1}/{2}}$, and assume $h-w$ is critical and
\begin{align*}
    h(u^{{1}/{2}})=\infty.
\end{align*}
Then for all finite $ W\subseteq X$ and all $\lambda>0$, the
inequality
\begin{align*}
    h- (1+\lm)w\ge 0
\end{align*}
fails to hold on $C_c(X\setminus W)$.
\end{thm}
We prove this theorem by assuming that the inequality in the theorem holds for some $ \lm>0 $ and show that contradicts the null-criticality. To this end, we need to show that we can extend the inequalities in question to  a larger class of functions. This is achieved by the following lemma
\begin{lemma}\label{lem_62}
Let $b,\ow b$ be graphs, $q, q'$ potentials, and let $h, h'$ be the corresponding forms with the
associated operators $\H$ and $\H'$. Let $v$ be a positive
$\H$-superharmonic function and suppose that $v^{1/2}$ is a positive $\H'$-harmonic function.
Assume $ h'$ is critical  on $X$ and that there is $C\ge 0$ such that
\begin{align*}
    h(\ph)\leq C h'(\ph), \quad  \ph\in C_c(X\setminus W),
\end{align*}
for $W\subseteq X$. Then
\begin{align*}
    h_{v}(v^{-1}f)\leq C  h_{ v^{1/2}}'( v^{-1/2}f), \qquad f\in C(X\setminus W).
\end{align*}
\end{lemma}
\begin{proof}
Since $ h'$ is critical on $X$, it follows that $ h_{v^{1/2}}'$
is critical as well by Corollary~\ref{c:gst}. Then by \cite[Theorem~5.3~(iv)]{KePiPo1}, there exists a the null-sequence  $(e_{n})$ for $h_{v^{1/2}}'$ such that
$e_{n}\in C_{c}(X)$, $0\leq e_{n}\leq 1$, $e_{n}\to 1$ and $  h_{
v^{1/2}}'(e_{n})\to 0$.

Applying the ground state transforms (Proposition~\ref{t:GST}), we see
\begin{align*}
h_{v}(v^{-1}\ph)\leq  h (\ph) \leq C  h'(\ph) =C h_{v^{1/2}}'(
v^{-1/2}\ph)
\end{align*}
for all $\ph\in C_{c}(X\setminus W)$, and hence, multiplying all
functions by $v^{1/2}$ yields
\begin{align*}
h_{v}(v^{-1/2}\ph)\leq C h_{v^{1/2}}'( \ph)\qquad \forall \ph\in C_{c}(X\setminus W).
\end{align*}

We next employ \cite[Lemma~5.11]{KePiPo1} which states that since $ h_{v^{1/2}}'$
is critical and  $(e_{n})$ is a null-sequence for $h_{v^{1/2}}'$ it follows that for every function $ f\in C(X) $ we have $ \lim_{n\to\infty}  h' (e_{n}f)=h'(f) $.
Hence, for $ f\in C(X)$ with support in $ X\setminus W $,  by Fatou's lemma and the  inequality in the assumption (noting that $ e_{n}f\in C_{c}(X\setminus W) $)   we obtain
\begin{align*}
h_{v}(v^{-1/2}f)\leq \liminf_{n\to\infty} h_{v}(v^{-1/2} e_{n}
f) \leq C\limsup_{n\to\infty} h_{v^{1/2}}'( e_{n}f)=C
h_{v^{1/2}}'(f).
\end{align*}
Replacing $f$ by $v^{-{1}/{2}}f$ yields
the claim of the lemma.
\end{proof}
We recall the following notation: Given
a graph $b$, zero potential $q=0$, and a strictly positive function $g$, the form $h_{g}$ on
$C(X)$ acts as
$$h_{g}(f)=\frac{1}{2}\sum_{x,y\in X}b(x,y)g(x)g(y)(f(x)-f(y))^{2},$$
which happens to coincide with the extension of the ground state
transform of $h$ to $C(X)$ whenever $g$ is $H$-harmonic.

\begin{proof}[Proof of Theorem~\ref{t:optimal}]

Throughout the proof $c$ and $C$ denote positive finite constants
that may change from line to line.

We set $w=(\H u^{1/2})/u^{1/2}$. Assume there is $\lm>0$ such that
$h-w\ge\lm w$ on $C_{c}(X\setminus W)$ for finite $W$. So,
\begin{align*}
h\leq C (h-w) %C_{c}(X\setminus W)
\end{align*}
on $ C_{c}(X\setminus W) $ with $  C=\frac{1+\lm}{\lm}$. We show that this leads to a
contradiction.

By definition of $w$ the function $u^{1/2}$ is $(\H-w)$-harmonic and
$u$ is $\H$-superharmonic by assumption. By Lemma~\ref{lem_62}, we
have
\begin{equation}\label{eq_14}
    h_{u}(u^{-1}f)\leq C (h-w)_{u^{1/2}}(u^{-{1}/{2}}f) \quad \forall f\in C(X\setminus W),
\end{equation}
 with the ground state transforms  $h_{u}$ of $h$
and $(h-w)_{u^{1/2}}$ of $h-w$.

Let us first estimate the left hand side of \eqref{eq_14} from below. Let
$f=u^{1/2}1_{X\setminus W}$. By the equality
$ab(a^{-1/2}-b^{-1/2})^{2}=(a^{{1}/{2}}-b^{{1}/{2}})^{2}$ applied
with $a=u(x)$, $b=u(y)$, we obtain
\begin{align*}
h_{u}(u^{-1}f)=h_{u}(u^{-1/2}1_{X\setminus W})= h_{1_{X\setminus W}}(u^{{1}/{2}})+\sum_{x\in W,y\in
X\setminus W}b(x,y)u(y),
\end{align*}
where we observe that the second term on the right hand side is
finite  since $u\in \F(X)$ and $W$ is finite. Furthermore,
\begin{align*}
\ldots&= h_{1}(u^{{1}/{2}}) -h_{1_{W}}(u^{{1}/{2}})-\!\!\!\sum_{x\in
W,y\in X\setminus W}
\hspace{-0.5cm}b(x,y)((u^{{1}/{2}}(x)-u^{{1}/{2}}(y))^{2}-u(y))\\
&\ge h_{1}(u^{{1}/{2}})-h_{1_{W}}(u^{{1}/{2}})-\sum_{x\in
W}u(x)\sum_{y\in X\setminus W}b(x,y),
\end{align*}
where  the second and the third term on the right hand side are
finite since $W$ is finite and $\sum_{y\in X}b(x,y)<\infty $ for all
$x$. Since $q\ge0$ is compactly supported and we assume that  $h(u^{{1}/{2}})=\infty$, it follows that $h_{1}(u^{{1}/{2}})=\infty$ and, thus, we conclude that
$$h_{u}(u^{-1}f)=\infty.$$

For the right hand side of \eqref{eq_14},  we get with
$f:=u^{1/2}1_{X\setminus W}$
\begin{align*}
     (h-w)_{u^{1/2}}(u^{-{1}/{2}}f)
    &=  (h-w)_{u^{{1}/{2}}}(1_{X\setminus W})\\
    & \leq \sum_{x\in W,y\in X\setminus W}b(x,y)(u(x)u(y))^{{1}/{2}}\\
    &\leq\left(\sum_{x\in W,y\in X\setminus W}\hspace{-.5cm}b(x,y)u(x)\right)^{{1}/{2}} \hspace{-.2cm}
    \left(\sum_{x\in W,y\in X\setminus W}\hspace{-.5cm}b(x,y)u(y)
    \right)^{{1}/{2}}\!\!.
\end{align*}
The right hand side is finite since $u\in \F(X)$, the set $W$ is
finite, and $\sum_{y\in X}b(x,y)<\infty $ for all $x$.

Thus, $( h-w)_{u^{1/2}}(u^{-{1}/{2}}f) <\infty$ while
$h_{u}(u^{-1}f)=\infty$ which is a contradiction to \eqref{eq_14}. This proves the
theorem by the discussion in the beginning of the proof.
\end{proof}
%%%%%%%%%%%%%%%%%
\section{Proof of the main theorem}\label{sec_pf_main}
%%%%%%%%%%%%%%%%%%%%%%%%%%%%%%
\begin{proof}[Proof of Theorem~\ref{thm_optimal_H}]
Let $u,v$ be positive $H$-superharmonic functions on $X$, such that
$u,v$ are $\H$-harmonic outside of a finite set. Let $u_{0}:=u/v$.
Then the positive function $u_{0}$ is
$H_{v}$-superharmonic  and $H_{v}$-harmonic outside of the finite
set. Let
\begin{align*}
w&:=\frac{{\H}_{v}\left[u_{0}^{{1}/{2}}\right]}{u_{0}^{{1}/{2}}}
=\frac{{\H}\left[(uv)^{{1}/{2}}\right]}{(uv)^{{1}/{2}}}.
\end{align*}
By Corollary~\ref{c:gst}, $ w $ is an optimal Hardy weight for $h$ if and only if $ w'=v^{2}w $ is an optimal Hardy weight for $h':= h_{v} $. Note that by Green's formula the form $ h'=h_{v} $ corresponds to the operator $ H'=L'+q' $ where  $L'$ is the operator associated to the graph $ b'(x,y) =b(x,y)v(x)v(y)$, $ x,y\in X $, and $ q'= v( Hv )$, i.e.,
\begin{align*}
h'(\ph,\psi)=\langle H'\ph,\psi\rangle_1.
\end{align*}
Note that by assumption on $ v $ the potential $ q' $ is finitely supported and since $ H'=T_{v^{2}}H_{v} $ the function $ u_{0} $ is a positive $ H' $-superharmonic function that satisfies the assumptions of Theorem~\ref{t:critical}, Theorem~\ref{t:nullcrit} and Theorem~\ref{t:optimal}. Hence, $ w'=\frac{H'(u_0^{1/2} )}{u_{0}^{1/2}}=v^{2}w$ is an optimal Hardy weight for $ h'=h_{v} $ which finishes the proof.  The explicit formula for $w$ follows from the
chain rule of the square root (see Lemma~\ref{l:squareroot} and
Corollary~\ref{c:squareroot}).
\end{proof}
%%%%%%%%%%%%%%%%%%%%
\begin{proof}[Proof of Corollary~\ref{thm_optimal_H_bounded}]
First of all, by the chain rule for the square root, the function $(u(v-u))^{1/2}$ is
$\H$-superharmonic, (see, Lemma~\ref{l:squareroot} and Corollary~\ref{c:squareroot}, where
also the explicit formula of $w$ can be read from). Let us check that
the assumptions of Corollary~\ref{thm_optimal_H_bounded} imply the assumptions of
Theorem~\ref{thm_optimal_H}. We let $u_{0}=u/v$ and
$$v_{0}:=\frac{u}{v-u}=\frac{u_{0}}{1-u_{0}}=\frac{1}{u_{0}^{-1}-1}.$$
Clearly,  $\sup u_{0}=1$ implies
\begin{align*}
        \sup v_{0}=\frac{1}{(\sup u_{0})^{-1}-1}=\infty.
\end{align*}
This gives assumption (c) of Theorem~\ref{thm_optimal_H}.

Let us check the validity of the assumption~(a) in Theorem~\ref{thm_optimal_H}.
Let $[a,b]\subseteq (0,\infty)$ and $x\in X$ such that $v_{0}(x)\in
[a,b]$. It follows that $u_{0}(x)\in [a/(a+1),b/(b+1)]\subset
(0,1)$. By assumption (a) of Corollary~\ref{thm_optimal_H_bounded}, there are only finitely many of these $x$
and hence Theorem~\ref{thm_optimal_H}~(a) follows. Assumption (b) of
Theorem~\ref{thm_optimal_H} follows directly from assumption
(b) of Corollary~\ref{thm_optimal_H_bounded}.
\end{proof}
Finally, we explain how  the two special cases in the Introduction can be derived from the
main theorems.

To this end we first prove the following lemma in the context of graphs with standard weights, i.e., $ b(x,y)\in\{0,1\} $, $ x,y\in X $.
\begin{lemma}\label{lemma7}
Assume that $ \deg(x) \leq C$ for all $ x\in X $, and let $ u $ be a positive $\Delta$-superharmonic on $ W\subseteq X $.
    Then
    \begin{align*}
    \sup_{\substack{x\sim y\\x\in W}} \frac{u(x)}{u(y)}\leq C.
    \end{align*}
    \end{lemma}
\begin{proof}
Since $u> 0$ and $\Delta u(y) \ge 0$, we get for $x\sim y$ 
    \begin{align*}
    u(x)\leq \sum_{z\sim y}u(z) \leq \deg(y)u(y) \leq C u(y)
    \end{align*}
    where $C>0$ does not depend on $x\in W$.
\end{proof}

\begin{proof}[Proof of Theorem~\ref{t:graph1}]
Let $ h_{X\setminus K} $ be the restriction of the  form  $ h $ to the space $ C_{c}(X\setminus K) $. Then
the operator $ H_{X\setminus K}$ acts as
\begin{equation*}
H_{X\setminus K}\ph(x)=
\sum_{y\in X \setminus K,\;y\sim x}(\ph(x)-\ph(y)) +q(x)\ph(x),
\end{equation*}
with $q(x):=\#\{z\in  K\mid z\sim x \}$.
Hence, $ v=1 $ is $ H_{X\setminus K} $-superharmonic in $ X\setminus K $
and $ H_{X\setminus K} $-harmonic outside of the combinatorial neighborhood of $ K $.
Moreover, as $ \Delta=H_{X\setminus K} $ for functions supported on $ X\setminus K $, the
restriction of $ u $ to $ X\setminus K $ is $ H_{X\setminus K} $-harmonic.

Assumption (a) of Theorem~\ref{thm_optimal_H}
is satisfied
for $ H_{X\setminus K} $ for $ u_0=u $.
Furthermore, assumption (b) follows from Lemma~\ref{lemma7}.
Hence, we obtain  for $ \ph\in C_{c}(X\setminus K) $,
\begin{align*}
\sum_{x\sim y}(\ph(x)-\ph(y))^{2}=h(\ph)\ge \sum_{x\in X\setminus K}w(x)\ph^{2}(x)
\end{align*}
with  optimal $ w $ given by
$$ w(x) = \frac{H_{X\setminus K} u^{1/2}}{ u^{1/2}}(x)= \frac{1}{2u(x)}
\sum_{y\sim x}\left(u(x)^{1/2}-u(y)^{1/2}\right)^{2}$$
for $ x\in X\setminus K $.
\end{proof}

\begin{proof}[Proof of Theorem~\ref{t:graph2}]
    We apply Theorem~\ref{thm_optimal_H} with $ v=G(o,\cdot) $ and $ u=1 $.
    In particular, assumption (a) of Theorem~\ref{thm_optimal_H}
    is satisfied for  $ u_0=1/G(o,\cdot) $.
    Furthermore, assumption (b) follows from the lemma above (Lemma~\ref{lemma7}).  Hence, the statement follows.
\end{proof}

\section{Examples}\label{s:examples}
\subsection{The $\Z^d$-case}
It is a well-known that for $d\geq 3$, the Green
function $G(x):= G(x,0)$ associated to the  Laplacian $\Delta$ on the standard
$\Z^d$-lattice has the following asymptotic behaviour (see
\cite[Theorem~2]{Uch98}, and the remark at the very end of
Section~2 therein):
%%%%%%%%%%%%%%%%%
\begin{thm}\label{GZd}
Let $d \geq 3$. Then as $|x| \to \infty$,
\[
G(x) = \frac{C_1(d)}{|x|^{d-2}} + C_2(d)\left(\left ( {\sum_{i=1}^d \Big(\frac{x_i}{|x|} \Big)^4}\right) - \frac{3}{d+2}\right )\frac{1}{|x|^d}
+ \mathcal{O}\left( \frac{1}{|x|^{d+2}} \right),
\]
where $C_1(d)$ and $C_2(d)$ are positive constants depending only on $d$.
\end{thm}
It follows from Theorem~\ref{t:graph2} that
$$w(x):=\frac{\Delta[G^{1/2}(x)]}{G^{1/2}(x)},\quad\,\, x\in \Z^d \setminus \{0\},$$
is an optimal Hardy weight for $\Delta$. We use Theorem~\ref{GZd} to
derive the asymptotic behavior of $w$ as $|x|\to \infty$.

\begin{thm}
Let $d \geq 3$. Then as $|x|\to\infty$ we have
\[
w(x)=\frac{\big( d-2 \big)^2}{4} \,\frac{1}{|x|^2} + \mathcal{O}\Bigg( \frac{1}{|x|^3} \Bigg).
\]
\end{thm}

\begin{proof}
Using Theorem~\ref{GZd}, one obtains for $|x| \to \infty$ and $y\sim x$ that
\[
\frac{G(y)}{G(x)} = \frac{|x|^{d-2}}{|y|^{d-2}} + \mathcal{O}\left( \frac{1}{|x|^3} \right).
\]
Keeping this observation in mind, it follows that
\begin{align*}
\sum_{y \sim x} \Bigg( \frac{G(y)}{G(x)} \Bigg)^{1/2}& =
\sum_{i=1}^d \sum_{\epsilon \in \{\pm 1\}} \Bigg( \frac{|x|^{d-2}}{\big(|x|^{2} +
\epsilon2\,x_i + 1\big)^{(d-2)/2}} \Bigg)^{\!1/2}\!\!+ \mathcal{O}\left( \frac{1}{|x|^3} \right)\\
&= \sum_{i=1}^d \sum_{\epsilon \in \{\pm 1\}} \Bigg( 1 + \frac{\epsilon 2\,x_i}{|x|^2} + \frac{1}{|x|^2} \Bigg)^{\!\!(2-d)/4}\!\!\!
+ \mathcal{O}\left( \frac{1}{|x|^3} \right).
\end{align*}
We now use the Taylor series expansion
\[
\big( 1 + z\big)^{\alpha} = 1 + \alpha\,z + \binom{\alpha}{2}z^2  + \mathcal{O}(|z|^3)
\]
for $z \in \R$ with $|z| < 1$ and $\alpha \in \R$ (here the generalized binomial coefficient
$\binom{\alpha}{2}$ is defined as $\binom{\alpha}{2} := \alpha\,\big( \alpha - 1 \big)/2$). As $x\to \infty$, this yields
\begin{align*}
\sum_{y \sim x} \left( \frac{G(y)}{G(x)} \right)^{\!\!\!1/2}&  =
\sum_{i=1}^d \left( 2 +  \frac{2(2-d)}{4}\,\frac{1}{|x|^2}  + 2\binom{\frac{2-d}{4}}{2}\,\frac{4x_i^2}{|x|^4} \right)
+ \mathcal{O}\left( \frac{1}{|x|^3} \right) \\[2mm]
&=
2d + \frac{d( 2-d )}{2} \frac{1}{|x|^{2}} + 8\,\frac{(2-d)(2-d-4)}{2\cdot 16 }\,\frac{1}{|x|^2} + \mathcal{O}\left( \frac{1}{|x|^3} \right)\\[2mm]
&= 2d - \frac{\big( d-2 \big)^2}{4} \,\frac{1}{|x|^2} +  \mathcal{O}\left( \frac{1}{|x|^3} \right).
\end{align*}
It follows that
\begin{equation*}
w(x) \!=\! \frac{\Delta\left[G^{1/2}(x)\right]}{G^{1/2}(x)}\!=\!  2d \!-\!\! \sum_{y \sim x} \left( \frac{G(y)}{G(x)} \right)^{1/2}=\frac{\big( d-2 \big)^2}{4} \,\frac{1}{|x|^2} +  \mathcal{O}\Bigg( \frac{1}{|x|^3} \Bigg).
\end{equation*}
has the claimed asymptotics as $|x| \to \infty$.
\end{proof}

\subsection{The half line}
In this subsection we show that we not only recover the classical
Hardy inequality \eqref{CHI}, but that we can also improve it, in the sense that we
can show lower order terms. In particular, in contrast to the continuous case, the operator associated with the classical Hardy inequality \eqref{CHI}
is subcritical in $\N$.
\begin{thm}[\cite{KePiPo3}] For all finitely supported functions $\ph:\N_{0}\to \R$

with $\ph(0)=0$ we have
\begin{align}\label{ICHI}
 \sum_{n=0}^{\infty}|\ph(n)-\ph(n+1)|^{2}\ge
\sum_{n=1}^{\infty}w(n)|\ph(n)|^{2}
\end{align}
with an optimal Hardy-weight $w$ given by
$$w(n)= \sum_{k=1}^\infty \binom{4k}{2k}\, \frac{1}{(4k-1)
\, 2^{(4k-1)} }\frac{1}{n^{2k}}=  \frac{1}{4 n^2}+ \frac{5}{64 n^4} +
\dotsb,$$
for $ n\ge 2 $ and $ w(1)=
2-\sqrt{2} $. In particular, $w(n)> \frac{1}{4 n^2}$ for any $n\geq 1$.
\end{thm}
\begin{proof} Consider the Laplacian (with standard weights) acting on $\N\subset \Z$ as
\begin{align*}
    \Delta u(n):= 2\ph(n)-\ph(n+1)-\ph(n-1),\qquad n\in \N.
    \end{align*}
Then the identity function $u(n):= n$ on $\N_0$ is positive  and harmonic on
$\N$. Moreover, by choosing $v(n):=1$  on $\N_0$ (which is harmonic on $\N$), it follows that the assumptions of
Theorem~\ref{t:graph1} are satisfied.
Hence, the corresponding optimal Hardy
weight $w$ is given by
\begin{align*}
 w(n) &= \frac{1}{2n} \left[
    \left(n^{1/2}-(n+1)^{1/2}\right)^{2}+
    \left(n^{1/2}-(n-1)^{1/2}\right)^2\right]\\
&=\frac{1}{n} \left[2n-
n^{1/2}\left((n+1)^{1/2}+(n-1)^{1/2}\right)\right]\\
&=2-\left(\left(1+\frac{1}{n}\right)^{1/2}
+\left(1-\frac{1}{n}\right)^{1/2}\right).
\end{align*}
Employing the Taylor expansion of the square root at $1$, i.e.,
    \begin{equation*}
\left(1\pm\frac{1}{n}\right)^{1/2} = \sum_{k=0}^\infty \binom{1/2}{k}\,
\left(\frac{\pm 1}{ n}\right)^{k}
    = 1 \pm
\frac{1}{2n} - \frac{1}{8n^{2}}\pm \frac{1}{16 n^{3}} - \frac{5}{128
n^{4}} \pm \dotsb
\end{equation*}
yields the result.
\end{proof}

\begin{center}{\bf Acknowledgments} \end{center}
The authors thank W.~Cygan and W.~Woess for pointing out the paper \cite{Uch98} and R.~Frank for helpful comments on the literature. M.~K. is grateful to the Department of Mathematics at the Technion for the hospitality during his visits and acknowledges the financial support of the German Science Foundation. Y.~P. and F.~P. acknowledge the support of the Israel Science Foundation (grants No. 970/15) founded by the Israel Academy of Sciences and Humanities. F.~P. is grateful for support through
a Technion Fine Fellowship.
%%%%%%%%%%%%%%%%%%%%%%%%%%%%%%%%%%%%%%%%%%%%%%%

%\bibliography{literature}
%\bibliographystyle{alpha}
\newcommand{\etalchar}[1]{$^{#1}$}

%%%%%%%%%%%%%%%%%%%%%%%%%%%%%
\end{document}